\pdfoutput=1
\documentclass[12pt,a4paper,english]{style}
\pagespan{1}{?}
\usepackage{babel}
\usepackage[utf8]{inputenc}
\usepackage[T1]{fontenc}
\usepackage{times}
\usepackage{graphicx}
\usepackage{lineno}

\newtheorem{theorem}{Theorem}[section]

\theoremstyle{remark}

\usepackage{times}
\usepackage{bbding}

\usepackage{hyperref}
\usepackage{float}
\usepackage{caption}
\usepackage{subfig}

\usepackage{tikz}
\usetikzlibrary{shapes,arrows}

\newcounter{minutes}
\divide\time by 60
\newcounter{hours}
\multiply\time by 60
\addtocounter{minutes}{-\time}
\newcommand{\C}{\mathbb{C}} % Kompleksiluvut
\newcommand{\D}{\mathbb{D}} % Unit Disk
\newcommand{\symD}{\Omega} % Domain
\begin{document}

\title[Harmonic Shears and Numerical Conformal Mappings]{Harmonic Shears and Numerical Conformal Mappings}
%\date{January 1, 2002}
\date{\today}

%delete this block if only one author, copy if more than two authors
%\author[S.~Ponnusamy]{Saminathan Ponnusamy}
%\email{samy@iitm.ac.in}
%\address{Indian Institute of Technology Madras,
%         Department of Mathematics,
%         Chennai-600 036,
%         India.}

%\author[T.~Quach]{Tri Quach${}^\textrm{{\tiny\EightFlowerPetal}}$}%\thanks{${}^\textrm{{\tiny\FiveFlowerOpen}}$}
\author[Tri Quach]{Tri Quach}
\email{tri.quach@aalto.fi}
\address{Aalto University,
         Department of Mathematics and Systems Analysis,
         P.O. Box 11100,
         FI-00076 Aalto,
         Finland.}


% fill out if necessary or keep empty, acknowledgements go before bibliography
\thanks{%This research was supported by a grant from the Jenny and Antti Wihuri Foundation. \\
The author was supported by a grant (ma2011n25) from the Magnus Ehrnrooth Foundation.}
%\dedicatory{Research is dedicated ...}

\begin{abstract}
In this article we introduce a numerical algorithm for finding harmonic mappings by using the {\it shear construction} introduced by Clunie and Sheil-Small in 1984. The MATLAB implementation of the algorithm is based on the numerical conformal mapping package, the Schwarz-Christoffel toolbox, by T.~Driscoll. Several numerical examples are given. In addition, we discuss briefly the minimal surfaces associated with harmonic mappings and give a numerical example of minimal surfaces.

%This is just another burger article where I shall study the use of the Schwarz-Christoffel toolbox in the use harmonic shear. 

%In this paper, we study harmonic mappings by using the {\it shear construction}, introduced by
%Clunie and Sheil-Small in 1984...
%We consider two classes of conformal mappings, each of which
%maps the unit disk $\D$ univalently onto a domain which is convex in the horizontal direction,
%and shear these mappings with suitable dilatations $\omega$.
%Mappings of the first class map the unit disk $\D$ onto  four-slit domains
%and mappings of the second class take $\D$ onto regular $n$-gons.
%In addition, we discuss the minimal surfaces associated with such harmonic mappings. Furthermore,
%illustrations of mappings and associated minimal surfaces are given by using {\sc Mathematica}.
\end{abstract}

\maketitle

%\bigskip
%\begin{center}
%\texttt{FILE:~\jobname .tex, 2009-11-17,
%        printed: \number\year-\number\month-\number\day,
%        \thehours.\ifnum\theminutes<10{0}\fi\theminutes}
%\end{center}

%\vspace*{-1cm}

%%%%%%%%%%%%%%%%%%%%%%%%%%%%%%%%%%%%%%%%%%%%%%%%%%%%%%%
%%%%%%%%%%%%%%%%%%%%%%%%%%%%%%%%%%%%%%%%%%%%%%%%%%%%%%%
%\linenumbers

\section{Introduction}

A complex-valued harmonic function $f=u+iv$, defined on the unit disk $\D$, is called a {\it harmonic mapping} if the coordinate function $u$ and $v$ are real harmonic, and it maps $\D$ univalently onto a domain $\symD \subset \C$. Note that it is not required that the real part and the imaginary part of $f$ are harmonic conjugate functions, i.e., satisfy the Cauchy-Riemann equations. In 1984, Clunie and Sheil-Small \cite{clunie} showed that many classical results for conformal mappings have natural analogues for harmonic mappings, and hence they can be regarded as a natural generalization of conformal mappings. Since then, this class of mappings has attracted considerable interest in complex analysis, see e.g. \cite{Dur}.

An important method for studying geometric properties of harmonic mappings is called harmonic shearing. Since its introduction in \cite{clunie}, it has been researched by many authors. For example, Greiner \cite{greiner} studied harmonic shears of conformal mappings from the unit disk onto infinite strips and other domains. Shearing of conformal mappings from the unit disk onto regular polygonal domains have been studied in the paper by Driver and Duren \cite{driver}, and by the author with Ponnusamy and Rasila \cite{pqr2}.

The shear construction makes use of a conformal mapping $\varphi$ and an analytic dilatation
$\omega$. For required assumptions for the dilatation $\omega$, see Section \ref{section: harmonic-mapping}. 

An applet \cite{rolf} for exploring harmonic shears with user defined conformal mappings is written by Rolf and examples are given in \cite{dorff}. However, no accuracy tests against analytic form have been given in \cite{dorff}. Note that, the applet requires an analytic expression for the conformal mapping. 

In this paper, the Schwarz-Christoffel toolbox by Driscoll \cite{driscoll} is used to provide a conformal mapping for the presented numerical method for harmonic shear. The toolbox computes a conformal mapping from the unit disk $\D$ onto a polygonal domain. See \cite{dt, trefethen} for details for the construction of a conformal mapping. 

In principle, other methods can be used to obtain conformal mappings as well. An example of a numerical method, that does not involve the Schwarz-Christoffel formula, is the Zipper algorithm of Marshall \cite{Mar, MR}. A method involving the harmonic conjugate function is presented in \cite[pp. 371-374]{Hen3}. 

An algorithm using the harmonic conjugate function and properties of quadrilateral is given by the author in a joint work with Hakula and Rasila \cite{hqr}. The algorithm is based on properties of the conformal modulus originating on the theory of quasiconformal mappings \cite{ahlfors,lv,ps}. The method is suitable for simply and doubly connected domains, which may have curved boundaries and even cusps. The implementation of the algorithm is based on the $hp$-FEM from \cite{hrv1}. %In \cite{hrv2}, the algorithm was studied in the case of unbounded domains.

An important application of harmonic mappings is related to minimal surfaces. A harmonic function $f = h + \overline{g}$ can be lifted to a minimal surface if and only if the dilatation $\omega$ is the square of an analytic function. Suppose that $\omega = q^2$ for some analytic function $q$ in the unit disk $\D$. Then the corresponding minimal surface has the form
\[
\{u,v,w\} = \{\textrm{Re}\, f,  \textrm{Im}\, f, 2\,\textrm{Im}\, \psi\},
\]
where
\[
\psi(z) = \int_0^z q(\zeta) \frac{\varphi(\zeta)}{1-\omega(\zeta)} \, d\zeta.
\]
The approach taken in this paper is based on the Weierstrass-Enneper representation. For further information about the relation between harmonic mappings and minimal surfaces can be found from \cite{Dur}. Minimal surfaces are closely related to many interesting phenomena in natural science and engineering, from mathematical models of soap bubble surfaces \cite{isenberg}, to topics in molecular engineering \cite{barth}, and tensile structures \cite{berger}. For a further reading on minimal surfaces and their applications see, e.g., \cite{nitsche,osserman}.

In this paper, the {\sc Matlab} visualisations of harmonic mappings are kept as close to the original images (cf. \cite{driver,pqr2}) as possible for easier comparison. Besides given illustrations, one can enhance the visualisations, e.g., by using phase portrait method \cite{wegert12, wegert13} or by using domain coloring method \cite{lundmark,poelke}.

%Discussion concerning the differential geometric approach to minimal surfaces and arose problem called Plateau's Problem can be found from the book by Pressley \cite{pressley}.

%%%%%%%%%%%%%%%%%%%%%%%%%%%%%%%%%%%%%%%%%%%%%%%%%%%%%%%
%%%%%%%%%%%%%%%%%%%%%%%%%%%%%%%%%%%%%%%%%%%%%%%%%%%%%%%

\section{Harmonic Mappings} \label{section: harmonic-mapping}

A harmonic mapping in $\D$
has a canonical presentation $f=h+\overline{g}$, where $h$ and $g$ are analytic in $\D$ and
$g(0)=0$. A harmonic mapping $f=h+\overline{g}$ is called {\it sense-preserving}
if the Jacobian $J_f = |h'|^2-|g'|^2$ is positive in $\D$. Then $f$ has an {\it analytic dilatation}
$\omega =g'/h'$ such that $|\omega(z)| < 1$ for $z\in\D$. For basic properties of harmonic
mappings, see \cite{Dur}.

A domain $\symD \subset \C$ is said to be {\it convex in the horizontal direction} (CHD)
if its intersection with each horizontal line is connected (or empty).
A univalent harmonic mapping is called a CHD mapping if its range is a CHD domain.
Construction of a harmonic mapping $f$ with prescribed dilatation
$\omega$ can be done by the {\it shear construction} devised by
Clunie and Sheil-Small \cite{clunie}. For reader's convenience, we recall the construction along with its basic properties.

\begin{theorem}
Let $f = h+\overline{g}$ be a harmonic and locally univalent in the unit disk $\D$.
Then $f$ is univalent in $\D$ and its range is a {\rm CHD} domain if and only if $h-g$ is a conformal
mapping of $\D$ onto a {\rm CHD} domain.
\end{theorem}

\begin{proof}
See, e.g., \cite[p. 37]{Dur}.
\end{proof}

Suppose that $\varphi$ is a CHD conformal mapping. For a given dilatation $\omega$, the harmonic shear
$f=h+\overline{g}$ of $\varphi$ is obtained by solving the differential equations
\[
\left\{ \begin{split}
h'-g' &=\varphi', \\
\omega h' - g' & = 0.
\end{split} \right.
\]
From the above equations, we obtain
\begin{equation} \label{eqn: fun-h}
h(z) = \int_0^z \frac{\varphi'(\zeta)}{1-\omega(\zeta)} \, d\zeta.
\end{equation}
For the anti-analytic part $g$, we have
\begin{equation} \label{eqn: fun-g}
g(z) = \int_0^z \omega(\zeta) \frac{\varphi'(\zeta)}{1-\omega(\zeta)} \, d\zeta.
\end{equation}
Observe that $f$ can also be written as
\begin{equation} \label{eqn: fun-f}
f(z)  = 2 \, \textrm{Re} \left[ \int_0^z \frac{\varphi'(\zeta)}{1-\omega(\zeta)} \, d\zeta \right] - \overline{\varphi(z)}.
\end{equation}
The last equation is useful if the conformal mapping $\varphi$ is known. 
%Note that, the anti-analytic part g of the harmonic mapping $f$ can be obtained from the identity $g = h - \varphi$ as well. A step-by-step algorithm can be given as follows:
%\begin{algorithm} (Harmonic Shearing) \label{alg: conf}
%\begin{enumerate}
%\item Choose a CHD conformal mapping $\varphi$.
%\item Choose a dilatation $\omega$.
%\item Compute $h$ and $g$ via \eqref{eqn: fun-h} and \eqref{eqn: fun-g}, respectively.
%\item Construct the harmonic mapping $f = h +\bar{g}$.
%\end{enumerate}
%\end{algorithm}

%%%%%%%%%%%%%%%%%%%%%%%%%%%%%%%%%%%%%%%%%%%%%%%%%%%%%%%
%%%%%%%%%%%%%%%%%%%%%%%%%%%%%%%%%%%%%%%%%%%%%%%%%%%%%%%

\section{Numerical Aspects}

Solving \eqref{eqn: fun-f} numerically, we shall use the change of variable $\zeta = zt$. Thus the analytic part of $f$ takes the following form
\begin{equation} \label{eqn: fun-h-int}
h(z) = \int_0^1 \frac{\varphi'(zt)}{1-\omega(zt)} z \, dt.
\end{equation} 
Then for the harmonic shear $f$, we have
\begin{equation} \label{eqn: fun-f-int}
f(z) = 2 \, \textrm{Re} \left[ \int_0^1 \frac{\varphi'(zt)}{1-\omega(zt)} z \, dt \right] - \overline{\varphi(z)}.
\end{equation}
Above integrals can be computed numerically, for example, by using the Gauss quadrature.

\subsection{Gauss Quadrature}
The key idea behind the Gauss quadrature is to choose the interpolation nodes in order to maximize the degree of exactness of the quadrature rule. In the Gauss quadrature one will consider integrals of the form
\begin{equation} \label{eqn: gauss-int}
I(g) = \int_a^b g(x)\, dx =  \int_a^b f(x)\eta(x)\, dx = \sum_{j=1}^N w_j f(x_j),
\end{equation}
where $\{ w_j, x_j \}_{j=1}^N$ is a quadrature rule corresponding to the weight function $\eta$. In \cite{golub-welsch}, Golub and Welsch gave algorithms to the Gauss quadrature for different weight functions. More in-depth discussion of the Gauss quadrature can be found, e.g., in \cite[Section 5.3]{gst}.

In our case, we have a singularity of the form $1/(1-\omega(z))$ caused by the dilatation. For some choices of dilatation, we can use the appropriate Gauss quadrature to deal with the singularity. However, in this article, we shall use the Gauss-Kronrod quadrature, which is also build in {\sc Matlab}. The Gauss-Kronrod quadrature is a generalization of a pure Gaussian quadrature and the method adds additional nodes to the Gauss rule with a way to control the error. It should also be noted that, the Gauss-Kronrod quadrature does not take possible singularities into account. For futher discussion of the Gauss-Kronrod quadrature, see \cite[pp. 299--300]{gst}.

%\begin{defn} (Weight Function) \cite[p. 132]{gst} \\
%The function $w(x)$ is called a {\it weight function} on the interval $[a,b]$ if it is non-negative in any open interval included in $[a,b]$ and
%\[
%\int_a^b |x|^n w(x) \, dx < \infty, \quad n = 0, 1, 2, \ldots .
%\]
%\end{defn}
%For computing \eqref{eqn: gauss-int}, one will approximate $\psi$ by polynomial $\psi_{n-1}$ of degree at most $n-1$, which interpolates $\psi$ at $n$ distinct nodes $x_1, x_2, \ldots, x_n$. That is,
%\[
%\psi(x) \approx \psi_{n-1}(x) = \sum_{j=1}^n \psi(x_j) L_j(x),
%\]
%where $L_j(x)$ are Lagrange interpolation polynomials defined by
%\[
%L_j(x) = \prod_{k=1, \, k\not=j}^n \frac{x-x_k}{x_j-x_k}.
%\]
%The approximation of the original integral is given by
%\[
%I(\psi) = \int_a^b f(x)w(x)\, dx \approx Q(\psi) = \sum_{j=1}^n w_j \psi(x_j),
%\]
%where
%\[
%w_j = \int_a^b L_j(x) w(x) \, dx.
%\]
%For more in-depth discussion of the Gauss quadrature, see \cite[Section 5.3]{gst}. For Gauss-Kronrod quadrature, see \cite[pp. 299--300]{gst}.

%%%%%%%%%%%%%%%%%%%%%%%%%%%%%%%%%%%%%%%%%%%%%%%%%%%%%%%
%%%%%%%%%%%%%%%%%%%%%%%%%%%%%%%%%%%%%%%%%%%%%%%%%%%%%%%

\subsection{Setup of the Validation Test}
Numerical experiments are divided into two meshing of the unit disk. For the first part, the following mesh of the form $re^{i\theta}$, where $r = \{ k/20: k=0,1,2,\ldots, 20 \}$ and $\theta = \{ 2\pi k/40: k=0,1,2,\ldots, 40\}$, is used. The second mesh is refined near the boundary of the unit disk, where $r = \{ (990+k)/1000: k=0,1,2,\ldots, 10\}$ and meshing for the angle is the same as in the first mesh.

Validation of the numerical scheme is run against the analytic representation of the harmonic mappings. All conformal mappings obtained by the SC toolbox are computed with the precision setting: \texttt{precision = 1e-14}. 

Note that, by using \eqref{eqn: fun-f}, one have two sources of errors. The first one comes from the integral  presentation of $h$ and the second source of error comes from the conformal mapping $\varphi$ itself.

In the unit disk, the comparison of the harmonic mapping is done by using the following test function
\begin{equation} \label{eqn: test}
\textrm{test} = |f - Q(f)|,
\end{equation}
where $f$ is the analytic expression of the harmonic mapping and $Q(f)$ is given by the quadrature. Note that, by \eqref{eqn: fun-f-int}, we may write the test function as follows
\begin{equation} \label{eqn: test-by-part}
\textrm{test} = |f - Q(f)| \leq 2\,| \textrm{Re}(h - Q(h))| + |\varphi - Q(\varphi)|,
\end{equation}
where $Q(h)$ and $Q(\varphi)$ are obtained by the Gauss-Kronrod quadrature and the SC toolbox, respectively.

%Also the component wise error of $h$ and $\varphi$ are reported in order to show the dominant part of the error.

%%%%%%%%%%%%%%%%%%%%%%%%%%%%%%%%%%%%%%%%%%%%%%%%%%%%%%%
%%%%%%%%%%%%%%%%%%%%%%%%%%%%%%%%%%%%%%%%%%%%%%%%%%%%%%%

\section{Examples of Polygonal Mappings}

In this section, we consider polygonal examples. Let the conformal mapping be (see \cite[p. 196]{nehari})
\[
\varphi(z) = \int_0^z (1-\zeta^n)^{-2/n} \, d\zeta,
\]
which maps the unit disk $\D$ onto a regular $n$-gon. In \cite{driver}, Driver and Duren discussed the harmonic shear of $\varphi$ with the dilatation $\omega (z) = z^n$. They studied other dilatation choices as well. The author considered the dilatation $\omega (z) = z^{2n}$ in a joint work with Ponnusamy and Rasila \cite{pqr2}.

\subsection{Analytic Form}
In \cite{driver}, it was shown that for the dilatation $\omega(z)=z^n$, the harmonic shear of $\varphi$ is given by
\[ \left\{ \begin{split}
h(z) & = z F\left(1+\frac{2}{n}, \frac{1}{n}; 1+\frac{1}{n}; z^n \right), \\
g(z) & = \frac{z^{n+1}}{n+1} F\left(1+\frac{2}{n}, 1+\frac{1}{n}; 2+\frac{1}{n}; z^n \right),
\end{split} \right.
\]
where $F(a,b;c;z)$ if the Gaussian hypergeometric function. The function is defined as follows
\[
F(a,b;c;z) = 1 + \sum_{n=1}^\infty \frac{(a)_n (b)_n}{n! (c)_n} z^n, \quad |z| < 1,
\]
where
\[
(\alpha)_n = \alpha (\alpha+1) \cdots (\alpha+n-1) = \frac{\Gamma(\alpha+n)}{\Gamma(\alpha)}, \quad \alpha \in \C,
\]
is the Pochhammer symbol. For $\textrm{Re}\, c > \textrm{Re}\, b > 0$, this can also be written as the Euler integral
\[
F(a,b,c;z) = \frac{\Gamma(c)}{\Gamma(b)\Gamma(c-b)} \int_0^1 t^{b-1} (1-t)^{c-b-1} (1-zt)^{-a} \, dt.
\]

In case of $\omega(z)=z^{2n}$, the harmonic shear is shown, in \cite{pqr2}, to be 
\[ \left\{ \begin{split}
h(z) & = z F_1\left( \frac{1}{n}, 1+\frac{2}{n}, 1; 1+\frac{1}{n}; z^n, -z^n \right), \\
g(z) & = \frac{z^{2n+1}}{2n+1} F_1\left( 2 +\frac{1}{n}, 1 +\frac{2}{n}, 1; 3 +\frac{1}{n}; z^n, -z^n \right),
\end{split} \right.
\]
where $F_1(a,b_1,b_2;c;x,y)$ is the first Appell hypergeometric function \cite[p. 73]{bailey}, which is defined by
\[
F_1(a,b_1,b_2;c;x,y) = \sum_{k=0}^\infty \sum_{l=0}^\infty \frac{(a)_{k+l} (b_1)_k (b_2)_l}{(c)_{k+l} \,k!\, l!} x^k y^l,
\]
Appell hypergeometric functions can be given by Euler's integral as follows \cite[p. 77]{bailey}:
\[
F_1(a,b_1,b_2;c;x,y) = \frac{\Gamma(c)}{\Gamma(a)\Gamma(c-a)} \int_0^1 t^{a-1} (1-t)^{c-a-1} (1-xt)^{-b_1} (1-yt)^{-b_2} \, dt,
\]
where $\textrm{Re}\, c > \textrm{Re}\, a > 0$. 

\subsection{Example}
Let $\omega(z)=z^{2n}$, $n=4$. The reason we chose this dilatation is that, in \eqref{eqn: fun-h} we have eight singularities and it reveals the accuracy of our algorithm and as well as the accuracy of the SC toolbox. In Figure \ref{fig: sc-poly4-dila-z8}, we have reproduced the figures from \cite{pqr2} and numerically computed images using the presented algorithm. 

\begin{figure}[!ht]
\begin{center}
\subfloat[Visualisation of an analytic form given by Mathematica.]{\parbox{\textwidth}{\centering\includegraphics[width=0.43\textwidth]{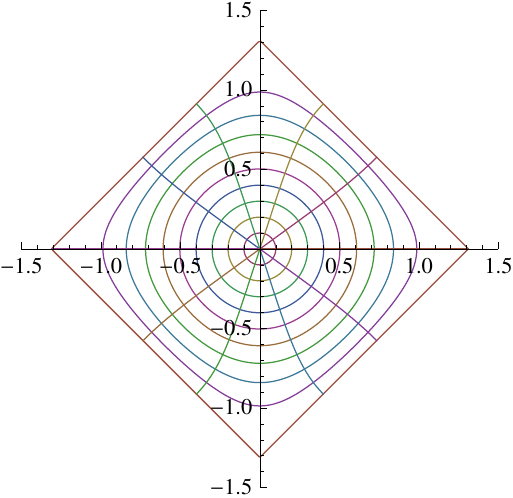} \hspace*{.3cm}
\includegraphics[width=0.43\textwidth]{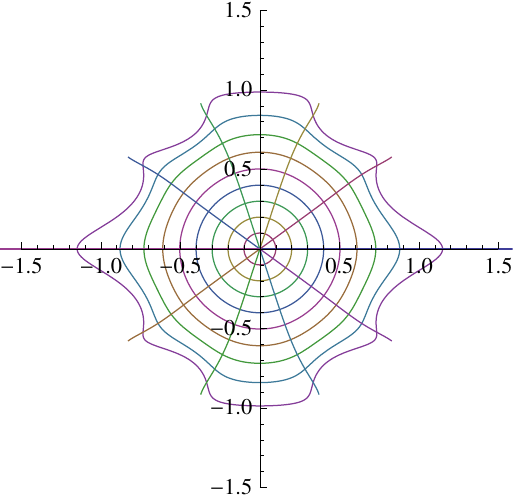}}} \\
\subfloat[A numerical version computed by {\sc Matlab}. Note that, the radius of the outermost circle is chosen to be 0.99, to avoid extensive amount of visual artefacts.]{\parbox{\textwidth}{\centering\includegraphics[width=\textwidth]{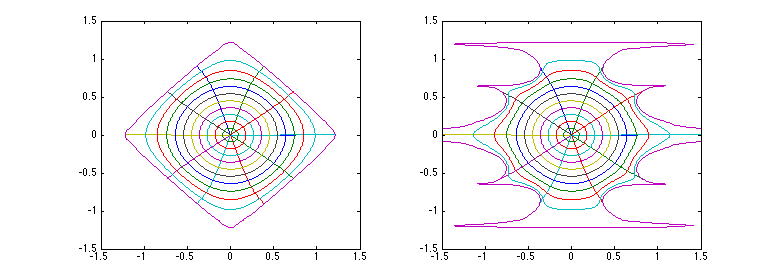} \hspace*{.3cm}}} \\
\caption{Conformal mapping $\varphi$ of the unit disk $\D$ onto a square and its harmonic shears with the dilatation $\omega(z) = z^{8}$.}  \label{fig: sc-poly4-dila-z8}
\end{center}
\end{figure}

In Figure \ref{fig: sc-poly-dila-z8-error-f}, we have the error of the test function \eqref{eqn: test} of the harmonic shear $f$ given in the logarithmic scale (with base 10). Note that, the error is given in the pre-image form, so that the figure can be shown in a more compact manner. The error corresponding to the second mesh can be given in polar coordinates as well, but the readability of the figure would be lost due to the poor resolution. Therefore, the error is given in Cartesian coordinates.

White areas in illustrations correspond to the loss of accuracy at singularities due to roots of the $1-z^8$. In the second mesh, near the boundary, we see that the rapid loss of accuracy occurs at the neighbourhood of the singularities.

\begin{figure}[!ht]
\begin{center}
\includegraphics[width=0.45\textwidth]{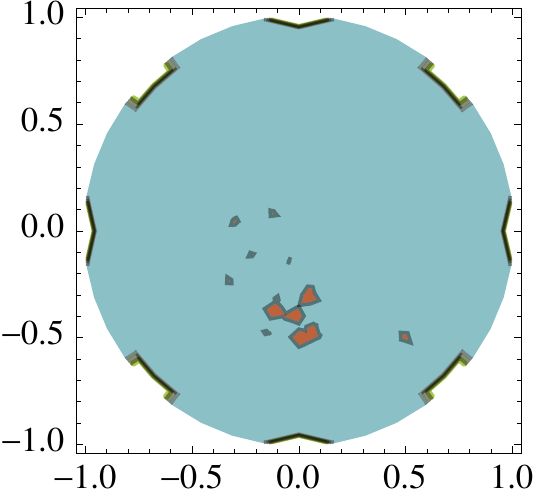} \hspace*{.2cm} 
\includegraphics[width=0.43\textwidth]{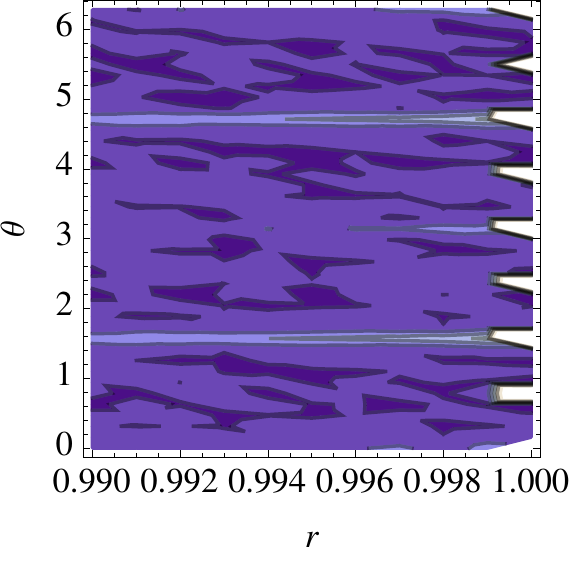} \\
\includegraphics[width=0.43\textwidth]{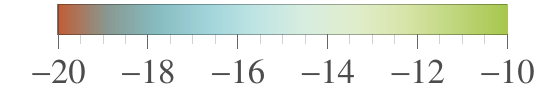} \hspace*{0.3cm} 
\includegraphics[width=0.43\textwidth]{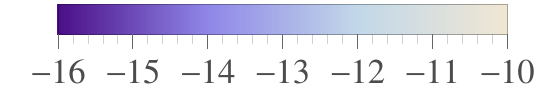}
\caption{Error analysis of the harmonic shear of conformal mapping $\varphi$ of the unit disk $\D$ onto a square with the dilatation $\omega(z) = z^{8}$. On the left hand side, we have the error of the first mesh, and on the right hand side, we the corresponding error near the boundary of the unit disk. Results are obtained by taking the logarithm (with base 10) of the test function \eqref{eqn: test}.} \label{fig: sc-poly-dila-z8-error-f}
\end{center}
\end{figure}

In Figure \ref{fig: sc-poly-dila-z8-error-h-phi}, we have errors corresponding to right hand side of the \eqref{eqn: test-by-part}. The conformal mapping itself has 4 singularities arising from the roots of $1-\zeta^4$. The loss of accuracy is around 4 digits in this particular example, which is nuisance but a not severe obstacle. 
Unfortunately, for the analytic part $h$ of $f$, even the leading digit is wrong at singularities most of the time. This can be improve if the type of singularities are known and the quadrature is adapted to take this information into account. Elsewhere the performance shows no significant shortcomings.

\begin{figure}[!ht]
\begin{center}
\subfloat[Error of the analytic part $h$ of $f$ in form of $2\,| \textrm{Re}(h - Q(h))|$ .]{\parbox{\textwidth}{\centering\includegraphics[width=0.45\textwidth]{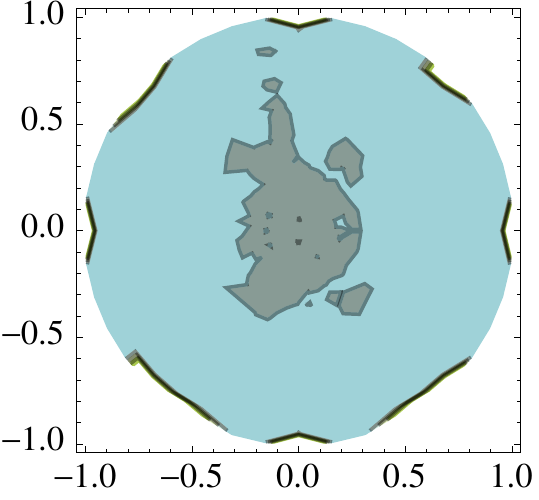} \hspace*{.2cm} 
\includegraphics[width=0.43\textwidth]{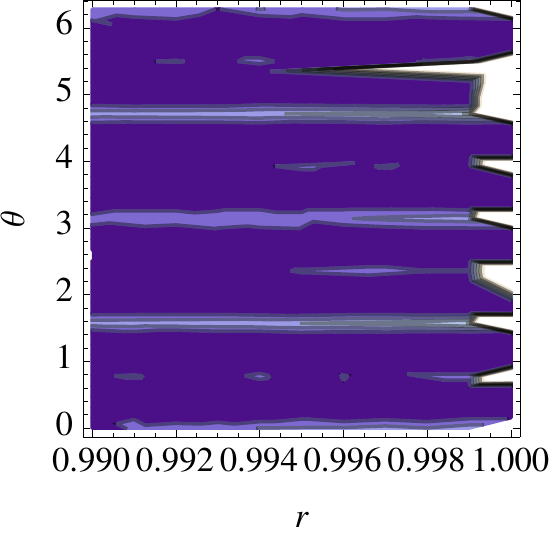} \\
\includegraphics[width=0.43\textwidth]{kuva/vertailu-sc-poly4-dila-z8-parempi-SC-legend-2014-04-03} \hspace*{0.3cm} 
\includegraphics[width=0.43\textwidth]{kuva/vertailu-sc-poly4-dila-z8-parempi-SC-legend-2014-04-03-reuna} }} \\
\subfloat[Error analysis of the conformal mapping $\varphi$ from the unit disk onto a square.]{\parbox{\textwidth}{\centering\includegraphics[width=0.45\textwidth]{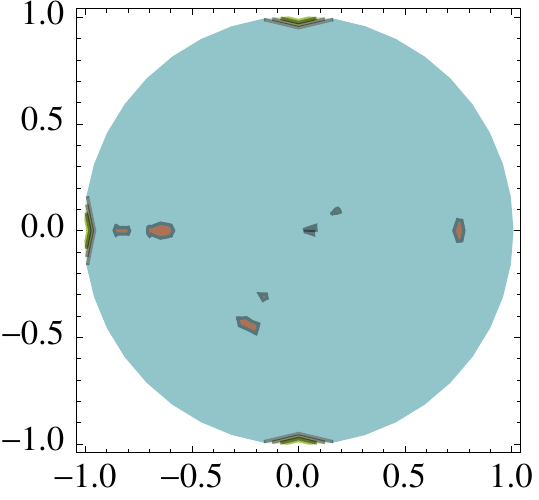} \hspace*{.2cm} 
\includegraphics[width=0.43\textwidth]{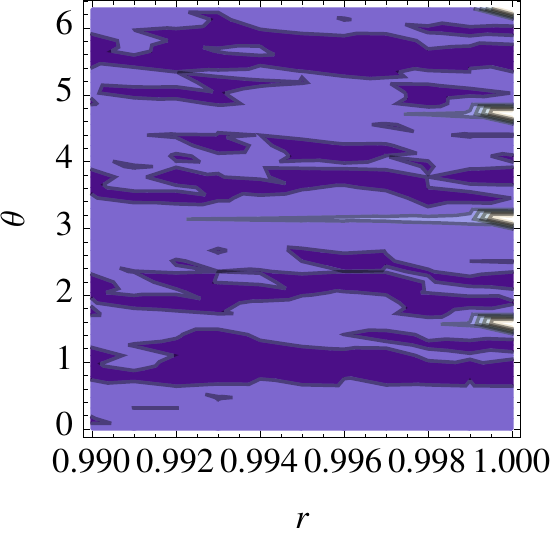} \\
\includegraphics[width=0.43\textwidth]{kuva/vertailu-sc-poly4-dila-z8-parempi-SC-legend-2014-04-03} \hspace*{0.3cm} 
\includegraphics[width=0.43\textwidth]{kuva/vertailu-sc-poly4-dila-z8-parempi-SC-legend-2014-04-03-reuna} }}
\caption{Error analysis of the analytic part $h$ of $f$ and the conformal mapping $\varphi$ in the case of the harmonic shear of conformal mapping $\varphi$ of the unit disk $\D$ onto a square with the dilatation $\omega(z) = z^{8}$. Results are obtained by taking the logarithm (with base 10) of the corresponding part of the test function \eqref{eqn: test-by-part}.} \label{fig: sc-poly-dila-z8-error-h-phi}
\end{center}
\end{figure}

\subsection{Polygonal Shears with Other Dilatations}
We shall reproduce the images from \cite{driver} for the dilation $\omega(z)=z^n$ and from \cite{pqr2} using dilatition $\omega(z)=z^{2n}$ with our numerical algorithm along with the error analysis. In Figure \ref{fig: sc-triangle-dila-z3-z6}, we have $n=3$. For illustrastions for $n=5$, see Figure \ref{fig: sc-pentagon-dila-z5-z10}.

\begin{figure}[!ht]
\begin{center}
\subfloat[Dilatation $\omega(z)=z^3$.]{\parbox{\textwidth}{\centering\includegraphics[width=0.44\textwidth]{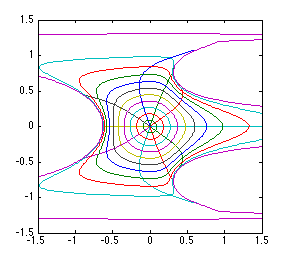} \hspace*{.1cm}
\includegraphics[width=0.4\textwidth]{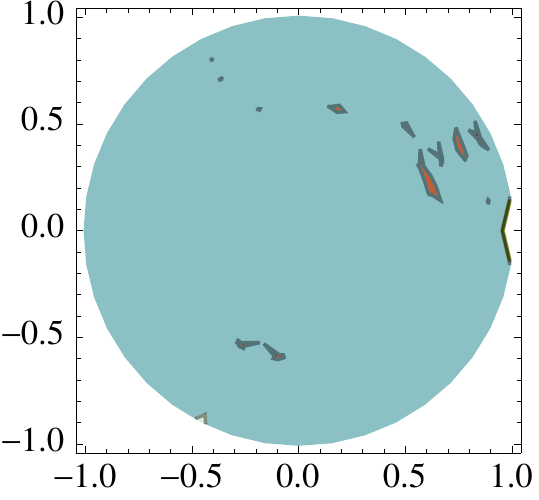} \hspace*{.1cm} 
\includegraphics[width=0.1\textwidth]{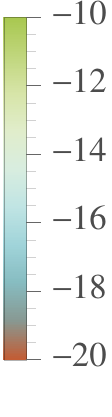} }} \\
\subfloat[Dilatation $\omega(z)=z^6$.]{\parbox{\textwidth}{\centering\includegraphics[width=0.44\textwidth]{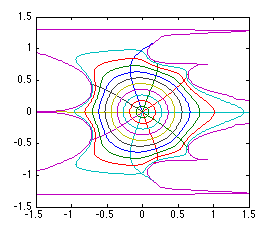} \hspace*{.1cm}
\includegraphics[width=0.4\textwidth]{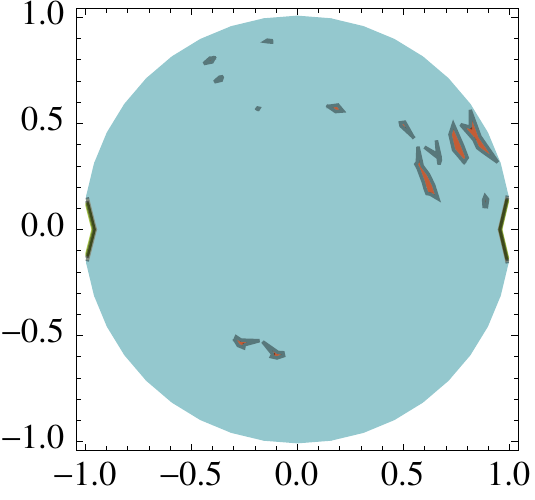} \hspace*{.1cm} 
\includegraphics[width=0.1\textwidth]{kuva/vertailu-sc-poly4-dila-z8-parempi-SC-legend-2014-04-14-vertical} }}
\caption{Reproduction from \cite{driver} and \cite{pqr2} of the conformal mapping $\varphi$ of the unit disk $\D$ onto a triangle and its harmonic shears with dilatation $\omega(z) = z^3, z^6$, respectively. The error is given in the logarithmic (base 10) scale. Note that, the radius of the outermost circle is chosen to be $0.99$, to avoid extensive amount of visual artefacts.}  \label{fig: sc-triangle-dila-z3-z6}
\end{center}
\end{figure}

\begin{figure}[!ht]
\begin{center}
\subfloat[Dilatation $\omega(z)=z^5$.]{\parbox{\textwidth}{\centering\includegraphics[width=0.44\textwidth]{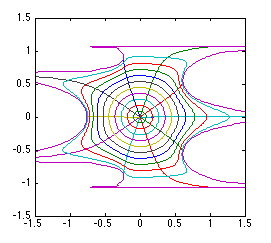} \hspace*{.1cm}
\includegraphics[width=0.4\textwidth]{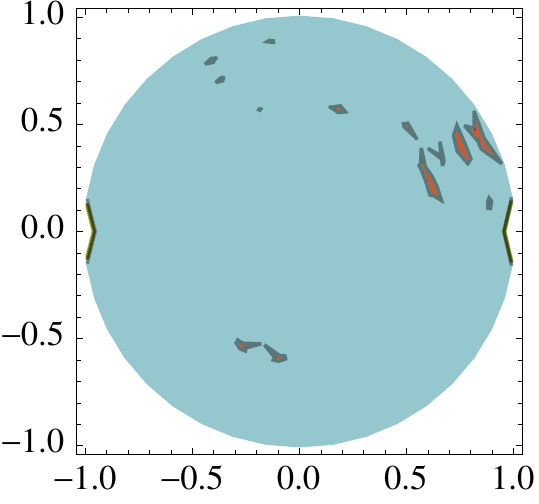} \hspace*{.1cm} 
\includegraphics[width=0.1\textwidth]{kuva/vertailu-sc-poly4-dila-z8-parempi-SC-legend-2014-04-14-vertical} }} \\
\subfloat[Dilatation $\omega(z)=z^{10}$.]{\parbox{\textwidth}{\centering\includegraphics[width=0.44\textwidth]{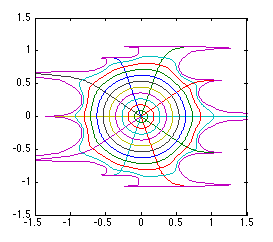} \hspace*{.1cm}
\includegraphics[width=0.4\textwidth]{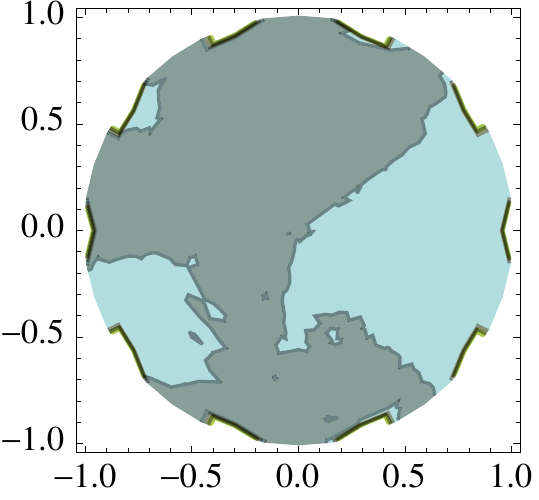} \hspace*{.1cm} 
\includegraphics[width=0.1\textwidth]{kuva/vertailu-sc-poly4-dila-z8-parempi-SC-legend-2014-04-14-vertical} }}
\caption{Reproduction from \cite{driver} and \cite{pqr2} of the conformal mapping $\varphi$ of the unit disk $\D$ onto a pentagon and its harmonic shears with dilatation $\omega(z) = z^5, z^{10}$, respectively. The error is given in the logarithmic (base 10) scale. Note that, the radius of the outermost circle is chosen to be $0.99$, to avoid extensive amount of visual artefacts.}  \label{fig: sc-pentagon-dila-z5-z10}
\end{center}
\end{figure}

%%%%%%%%%%%%%%%%%%%%%%%%%%%%%%%%%%%%%%%%%%%%%%%%%%%%%%%
%%%%%%%%%%%%%%%%%%%%%%%%%%%%%%%%%%%%%%%%%%%%%%%%%%%%%%%

\section{Minimal Surfaces}
It is known that a harmonic function $f = h + \overline{g}$ can be lifted to a minimal surface if and only if the
dilatation $\omega$ is the square of an analytic function. Suppose that $\omega = q^2$ for some analytic
function $q$ in the unit disk $\D$. Then the corresponding minimal surface has the form
\[
\{u,v,w\} = \{\textrm{Re}\, f,  \textrm{Im}\, f, 2\,\textrm{Im}\, \psi\},
\]
where
\[
\psi(z) = \int_0^z q(\zeta) \frac{\varphi(\zeta)}{1-\omega(\zeta)} \, d\zeta.
\]
In \cite{driver}, Driver and Duren computed $\psi$ for a conformal mapping $\varphi$ from unit disk onto a $n$-gon with dilatation $\omega(z)=z^n$. Computation is done with assumption that $n$ is even. In this case the minimal surface lifting is given by
\[
\psi(z) = \frac{2 z^{1+1/n}}{n+2} F\left(1+\frac{2}{n}, \frac{1}{2}+\frac{1}{n}; \frac{3}{2}+\frac{1}{n}; z^n \right).
\]
In Figure \ref{fig: square-minimal}, we have a illustration for $n=4$ given by the exact solution and the numerical scheme.

\begin{figure}[!ht]
\begin{center}
\includegraphics[width=0.4\textwidth]{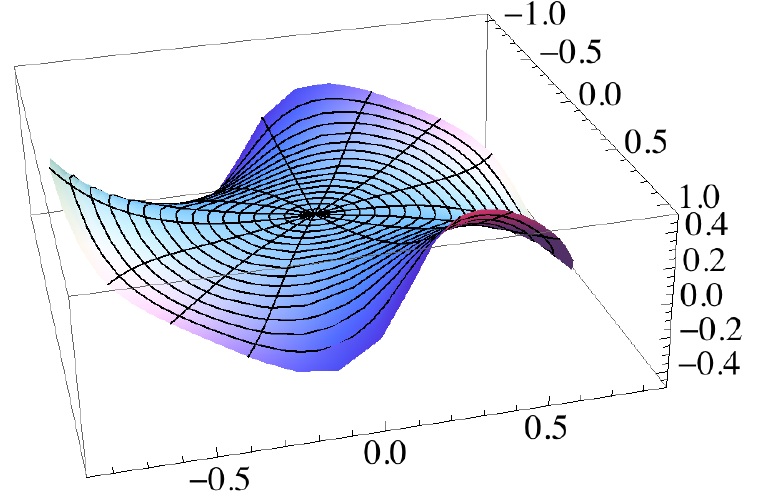} \hspace*{.1cm} 
\includegraphics[width=0.4\textwidth]{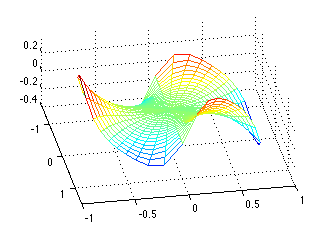}
\caption{Minimal surface of the conformal mapping $\varphi$ of unit disk onto $n$-gon with a dilatation $\omega(z)=z^n$, $n=4$. On the left hand side, we have a illustration given by the exact solution. In comparison, on the right hand side, we have a numerically computed version of the minimal surface. To give a sensible illustration, we have chosen the outermost circle's radius to be $0.8$. This example is a reproduction from \cite{driver}. }  \label{fig: square-minimal}
\end{center}
\end{figure}

%%%%%%%%%%%%%%%%%%%%%%%%%%%%%%%%%%%%%%%%%%%%%%%%%%%%%%%

\section{Conclusions}
In this article, we have given an algorithm to numerically shear CHD conformal mappings. Required integrations are done using a standard Gauss-Kronrod quadrature. From given examples against analytic expression, we found that our numerical method's performance is satisfactory on all mesh points we have considered except singularities. The accuracy can be improved if the type of the singularity is known and the algorithm is tweaked to work around it. 

%Even though the examples considered only shearing polygonal conformal mappings, the algorithm itself is general and can be used to all CHD conformal mappings. 

%%%%%%%%%%%%%%%%%%%%%%%%%%%%%%%%%%%%%%%%%%%%%%%%%%%%%%%

%\begin{acknowledgement}
%test
%\end{acknowledgement}
%\clearpage

%%%%%%%%%%%%%%%%%%%%%%%%%%%%%%%%%%%%%%%%%%%%%%%%%%%%%%%

%\clearpage

\end{document}